\documentclass[10pt,reqno]{amsart}
\usepackage{amsmath,amsopn,amssymb,amsthm}
\usepackage[cp1251]{inputenc}
\usepackage[english]{babel}
\usepackage[pagebackref,breaklinks=true,colorlinks=true,linkcolor=blue,citecolor=blue,urlcolor=blue]{hyperref}
\usepackage{graphicx}%
\usepackage{color}

\newtheorem{theorem}{Theorem}[section]

\newtheorem{example}[theorem]{Example}

\newtheorem{lemma}[theorem]{Lemma}

\newtheorem{remark}[theorem]{Remark}

\renewenvironment{proof}[1][Proof]{\noindent\textbf{#1.} }{\ \rule{0.5em}{0.5em}}

\newcommand{\m}{{\mathfrak m}}

\newcommand{\g}{{\mathfrak g}}
\newcommand{\h}{{\mathfrak h}}

\newcommand{\V}{{\mathbf V}}
\begin{document}

\title[Very standard homogeneous Finsler manifolds with positive flag curvature]{Very standard homogeneous Finsler manifolds with positive flag curvature}
\author{Xiyun Xu}
\address[Xiyun Xu] {School of Mathematical Sciences,
Capital Normal University,
Beijing 100048,
P.R. China}
\email{2210502058@cnu.edu.cn}
\author{Ming Xu}
\address[Ming Xu] {Corresponding author, School of Mathematical Sciences,
Capital Normal University,
Beijing 100048,
P.R. China}
\email{mgmgmgxu@163.com}

\begin{abstract}
In this paper, we consider a homogeneous manifold $G/H$ in which $G$ is
a compact connected simply connected simple Lie group and $H$ is a closed connected subgroup of $G$. We define standard and very standard homogeneous Finsler metrics on $G/H$, which generalize the standard
homogeneous $(\alpha_1,\alpha_2)$ metric in literature. We classify all
these $G/H$ which admit positively curved very standard homogeneous Finsler
metrics.

\vbox{}
\end{abstract}

\maketitle

\section{Introduction}

Homogeneous Riemannian manifolds which have positive sectional curvature have been classified during the nineteen sixties and seventies \cite{AW1975}\cite{Be1961}\cite{Be1976}\cite{Wa1972} (see
\cite{Wi1990}\cite{WZ2018}\cite{XW2015} for some refinement). It suggests us explore a similar project in Finsler geometry:
{\it Classifying smooth coset spaces $G/H$ which admits homogeneous Finsler metrics with positive flag curvature}.

In recent years, there have been many progresses in classifying positively curved homogeneous Finsler manifolds  \cite{XD2017}\cite{XD2017-2}\cite{XDHH2015}\cite{XZ2017} (a survey for these works can be found in \cite{DX2017}). It should be notified that, the classification is incomplete when the dimension is odd, especially, for non-reversible metrics. In this paper, we study this classification, which is restricted to {\it very standard homogeneous
Finsler metrics}.

A (very) standard homogeneous Finsler metric is defined in Section \ref{subsection-very-standard}. It generalizes the standard homogeneous $(\alpha_1,\alpha_2)$ metric in \cite{ZX2022}. Notice that a standard homogeneous $(\alpha_1,\alpha_2)$ metric must be reversible, but a general (very) standard homogeneous Finsler metric may not. A very standard homogeneous metric may be viewed as a
deformation for the
Riemannian normal homogeity, which preserves the most
computability. In Riemannian geometry, it has been extensively applied and explored. In Finsler geometry, it differs significantly to the normal homogeneity, because the latter one preserves instead some other features, for example, vanishing S-curvature, non-negative flag curvature, and geodesic orbit properties \cite{XD2017}.

The main theorem of this paper is as follows.

\begin{theorem}\label{main-thm}
Let $G/H$ be a smooth coset space in which $G$ is a compact connected simply connected simple Lie group, and $H$ is a closed connected
subgroup. Then $G/H$ admits very standard homogeneous Finsler metrics with
positive flag curvature if and only if it admits positively curved homogeneous Riemannian metrics.
\end{theorem}

Theorem \ref{main-thm} provides the following list of positively curved very standard homogeneous Finsler manifolds:
\begin{enumerate}
\item the compact rank one symmetric spaces, i.e. $S^n=Spin(n+1)/Spin(n)$, $\mathbb{C}\mathrm{P}^{n-1}=SU(n)/S(U(n-1)U(1))$, $\mathbb{H}\mathrm{P}^{n-1}=Sp(n)/Sp(n-1)Sp(1)$, with the integer $n>1$, and the Caylay plane $F_4/Spin(9)$;
\item the homogeneous spheres $S^{2n+1}=SU(n+1)/SU(n)$ and $S^{4n-1}=Sp(n)/Sp(n-1)$ with the integer $n>0$,  $S^6=G_2/SU(3)$, $S^7=Spin(7)/G_2$ and $S^{15}=Spin(9)/Spin(7)$, and the homogeneous complex projective spaces $\mathbb{C}\mathrm{P}^{2n-1}=Sp(n)/Sp(n-1)U(1)$ with the integer $n>0$;
\item the Berger spaces $Sp(2)/SU(2)$ and $SU(5)/Sp(2)U(1)$;
\item the Wallach spaces $SU(3)/U(1)^2$, $Sp(3)/Sp(1)^3$ and $F_4/Spin(8)$;
\item the Aloff Wallach spaces $SU(3)/S^1_{k,l}$ with $k,l\in\mathbb{Z}$ satisfying $kl(k+l)\neq0$, in which $S^1_{k,l}=\{\mathrm{diag}(z^k,z^l,z^{-k-l}),\forall z\in\mathbb{C},|z|=1\}$.
\end{enumerate}
This list is complete, and essentially the same as those in \cite{DX2017}\cite{WZ2018}, because each connected positively curved homogeneous manifold admits the transitive isometric action of
a compact connected simply connected simple Lie group.
Theorem \ref{main-thm} provides
another generalization of Berger's classification for positively curved normal homogeneous Riemannian manifolds \cite{Be1961}, besides \cite{XD2017}.

The theme for proving Theorem \ref{main-thm} is as follows. Besides the normal homogeneous Riemannian metrics in \cite{Be1961}, the positively curved homogeneous metrics constructed in \cite{AW1975}\cite{Wa1972} are in deed very standard. This observation proves one side of the theorem. To prove the other side, we divide the discussion into three cases, the even dimensional case, the reversible odd dimensional case,
and the non-reversible odd dimensional case. The even dimensional case follows immediately after the classification in \cite{XDHH2015}.
The reversible odd dimensional case is discussed in Section \ref{sub-section-the-reversible case}, where Theorem 1.1 in \cite{XZ2017} is used. When we verify case by case the five undetermined candidates in that theorem can not be positively curved, the very standard property brings us more orthogonality, so that the homogeneous curvature formula (see Theorem \ref{theorem 4.1}, or Theorem 4.1 in \cite{XDHH2015}) can be more conveniently applied.
The non-reversible odd dimensional case is discussed in Section \ref{sub-section-non-reversible}. We can find
a Finsler submersion $G/H\rightarrow G/K$ with $\dim G/K=\dim G/H -1$. From the classification of the even dimensional positively curved $G/K$, we can determine the corresponding $G/H$.

\section{Preliminaries}
In this section, we recall some knowledge in general and homogeneous Finsler geometry, which are necessary for later discussion. Most notions can be found in \cite{BCS2000} and \cite{De2012} respectively. In the last subsection, we define standard and very standard homogeneous Finsler metrics.

\subsection{Minkowski norm and Finsler metric}

A {\it Minkowski norm} on a real vector space $\mathbf{V}$ with $\dim\mathbf{V}=n$ is a continuous function $F: \V\rightarrow [0,+\infty)$ satisfying the following conditions:
\begin{enumerate}
\item $F$ is positive and smooth when restricted to $\V\setminus \{0\}$;
\item $F$ is positively homogeneous of degree one, i.e. $F(\lambda u)=\lambda F(u)$, for all $u\in\V$ and $\lambda>0$.
\item $F$ is strongly convex. Namely, choose any basis $\{e_1,e_2,\cdots,e_n\}$ of $\V$ and write $F(y)=F(y^1,y^2,\cdots,y^n)$ for $y=y^ie_i\in\V$, then the Hessian matrix
$(g_{ij}(y))=\left(\left[\frac{1}{2}F^2\right]_{y^iy^j}\right)$
is positive definite when $y\neq 0$.
\end{enumerate}
The Hessian matrix in (3) provides the {\it fundamental tensor}
$$g_y(u,v)=g_{ij}(y)u^iv^j=\frac12\frac{\partial^2}{\partial s\partial t}|_{s=t=0}F^2(y+tu+sv),\quad\forall u=u^ie_i, v=v^je_j\in\V,$$
which is an inner product on $\V$, parametrized by $y\in\V\backslash\{0\}$.
We call $F$ {\it reversible} if $F(y)=F(-y)$ is valid everywhere.

A {\it Finsler metric} on a smooth manifold $M$ is a continuous function $F: TM \rightarrow [0, \infty)$ such that
\begin{enumerate}
\item $F$ is smooth on the slit tangent bundle $TM\setminus0$;
\item The restriction $F(x,\cdot)$ of $F$ to each tangent space $T_xM$ is a Minkowski norm.
\end{enumerate}
When the manifold $M$ is endowed with a Finsler metric, we call $(M,F)$ a {\it Finsler manifold} or a {\it Finsler space}. The fundamental tensor and reversibility
of a Finsler metric are defined through each $F(x,\cdot)$.

\subsection{Riemann curvature and flag curvature}

For a Finsler manifold, the {\it Riemann curvature} $R_y:T_xM\rightarrow T_xM$ for $y\in T_xM\backslash\{0\}$ naturally appears in the Jacobi equation
for a variation of constant speed geodesics. In local coordiates, it can be presented as $R_y=R_k^i(y)\partial_{x^i}\otimes dx^{k}:T_{x}M\to T_{x}M$, in which
$$R_{k}^{i}(y)=2\partial_{x^{k}}G^{i}-y^{j}
\partial_{x^{j}y^{k}}^{2}G^{i}+2G^{j}\partial_{y^{j}y^{k}}^{2}G^{i}-\partial_{y^{j}}G^{i}\partial_{y^{k}}G^{j},$$
Here $G^i=\frac14g^{il}([F^2]_{x^ky^l}y^k-[F^2]_{x^l})$ is the geodesic spray coefficient.

Sectional curvature in Riemannian geometry can be generalized to {\it flag curvature} as follows. Let $y$ be a nonzero tangent vector in $T_xM$ and $\textbf{P}\subset T_xM$ a tangent plane containing $y$. Assuming $\mathbf{P}=\mathrm{span}\{y,v\}$, we call the triple $(x,y,\mathbf{P})=(x,y,y\wedge v)$ a {\it flag} of $M$, and define its {flag curvature} by
\begin{equation*}
K(x,y,\mathbf{P})=K(x,y,y\wedge v)=\frac{g_y(R_{y}v,v)}{g_y (y,y) g_y(v,v)-g_y(y,v)^{2}}.
\end{equation*}


\subsection{Homogeneous Finsler manifold and homogeneous curvature formula}
Let $(M,F)$ be a connected Finsler manifold. We denote by $I(M,F)$ its isometry group and by $I_0(M,F)$ the identity component of $I(M,F)$. We call $(M,F)$ {\it homogeneous} if $I_0(M,F)$ acts transitively on $M$. Since $I_0(M,F)$ is a Lie transformation group, we may present $M$ as a smooth manifold $G/H$, in which $G$ is a closed connected Lie subgroup of $I_0(M,F)$, and then $H$ is a compact subgroup of $G$.
In the Lie algebraic level, we have
a {\it reductive decomposition} for $G/H$, i.e., an $\mathrm{Ad}(H)$-invariant linear decomposition $\mathfrak{g}=\mathfrak{h}+\mathfrak{m}$, in which $\g=\mathrm{Lie}(G)$ and
$\h=\mathrm{Lie}(H)$.
The subspace $\mathfrak{m}$ can be canonically identified with the tangent space $T_o(G/H)$ at $o=eH$, such that the $\mathrm{Ad}(H)$-action on $\mathfrak{m}$ coincides with the isotropic $H$-action on $T_o(G/H)$.

To construct a homogeneous Finsler metric, we may directly start with a
homogeneous manifold $G/H$ with a compact $H$ and a reductive decomposition
$\g=\h+\m$. Given any $\mathrm{Ad}(H)$-invariant Minkowski norm $F$ on $\m=T_o(G/H)$,
we may use the $G$-actions to transport it to other points and get a homogeneous Finsler
metric on $G/H$, which is still denoted by $F$ for simplicity. This correspondence is one-to-one.

All curvatures of a homogeneous Finsler manifold $(G/H,F)$ can be described
explicitly by the Minkowski norm $F=F(o,\cdot)$ and the algebraic information
in the given reductive decomposition \cite{Hu2015}\cite{Xu2022}. For example, we have (see Theorem 4.1 in \cite{XDHH2015})

\begin{theorem}\label{theorem 4.1}
Let $(G/H, F)$ be a homogeneous Finsler manifold with a given reductive decomposition $\g = \h+\m$, $u$ and $v$ a linearly independent commuting pair in $\m$ satisfying $g_y(y,[y,\m]_\m)=0$.  we have
\begin{equation}\label{homogeneous flag curvature formula}
K(o,y,y\wedge v)=\frac{g_y( U(y,v),U(y,v))}{g_y (y,y) g_y( v,v)-
g_y(y,v)^2},
\end{equation}
where $U(y, v)\in\m$ is determined by
\begin{center}
$g_y(U(y,v),w)=\frac{1}{2}(g_y([w,y]_{\m},v)+g_y([w,v]_{\m},y))$, for any $w\in\m$.
\end{center}
Here the subscript $\m$ means is the linear projection to $\m$, with respect to the given reductive decomposition.
\end{theorem}

In this paper, we will use the following immediate corollary of Theorem
\ref{theorem 4.1}.

\begin{lemma}\label{lemma-1}
Keeping all assumptions and notations in Theorem \ref{theorem 4.1},
then $K(o,u,u\wedge v)=0$ when the linearly independent pair $u,v\in\m$ satisfy $$[u,v]=0\mbox{ and }g_u(u,[u,\mathfrak{m}]_\m)=g_u(u,[v,\mathfrak{m}]_\m)=g_u(v,[u,\m]_\m)=0.$$
\end{lemma}
\subsection{Very standard homogeneous Finsler metric}
\label{subsection-very-standard}
Throughout this paper, we will only discuss the homogeneous manifold $G/H$, in which $G$ is compact connected simply connected simple Lie group, and $H$ is a connected closed subgroup in $G$.

Denote by $\langle\cdot,\cdot\rangle$ a bi-invariant inner product on $\mathfrak{g}$, which is a suitable negative scalar of the Killing form.
With respect to $\langle\cdot,\cdot\rangle $,
$G/H$ has an orthogonal reductive decomposition $\g=\h+\m$. Further more, $\mathfrak{m}$ can be linearly decomposed as $\mathfrak{m}=\mathfrak{m}_1+\cdots+\mathfrak{m}_s$,
which is invariant with respect to the $\mathrm{Ad}(H)$-actions, and orthogonal with respect to $\langle\cdot,\cdot\rangle $. Then we have a standard block diagonal $SO(\m_1)\times\cdots\times SO(\m_s)$-action on $\m$.

We call a homogeneous Finsler metric $F$ on $G/H$ {\it standard} if
$\m$ has a decomposition as previously described, such that
the Minkowski norm $F=F(o,\cdot)$ on $\mathfrak{m}$ is invariant for
the corresponding $SO(\m_1)\times\cdots\times SO(\m_s)$-action. More over,
if in each $\m_i$, only one equivalent class of irreducible sub-$H$-presentations appear, and the irreducible sub-$H$-representations in $\m_i$ are
not equivalent to those in $\m_j$, whenever $i\neq j$, we call $F$
{\it very standard}.

\begin{remark}
The notion of standard homogeneous Finsler metric generalizes that of homogeneous $(\alpha_1,\alpha_2)$ metric in \cite{ZX2022}.
Generally speaking, $G/H$ may have different
decompositions $\m=\m_1+\cdots+\m_s$, and correspondingly, it has different families of standard homogeneous Finsler metrics. However,
the decomposition $\m=\m_1+\cdots+m_s$ for
a very standard homogeneous metric is unique. This decomposition is obviously $\langle\cdot,\cdot\rangle $-orthogonal by Schur Lemma.
\end{remark}

To present a very standard homogeneous metric $F$,
we borrow notations from $(\alpha,\beta)$ metrics and $(\alpha_1,\alpha_2)$.
For example, when $\dim\m_i\geq 2$ for each $i$, we may present
$F$ as
\begin{equation}\label{001}
F(y)=\sqrt{L(\langle y_1,y_1\rangle ,\cdots,\langle y_s,y_s\rangle )},
\end{equation}
in which $y=y_1+\cdots+y_s$, with $y_i\in\m_i$ for each $i$ (same below).
Obviously, $F$ is reversible in this case.
When $\dim \m_i>1$ for each $i<s$ and $\dim \m_s=1$, we may present
$F$ as
\begin{equation}\label{002}
F(y)=\sqrt{L(\langle y_1,y_1\rangle  ,\cdots,\langle y_{s-1},y_{s-1}\rangle ,y_s)}.
\end{equation}
Since the the trivial $SO(\m_s)$ does not impose any symmetry requirement,
The metric in (\ref{002}) may be non-reversible in this case.

We end this section by some examples.

\begin{example}\label{example-1}
A normal homogeneous Riemannain metric on $G/H$ is very standard.
On the three Wallach spaces,
$SU(3)/U(1)^2$, $Sp(3)/Sp(1)^3$ and $F_4/Spin(8)$, all homogeneous Riemannian metric are very standard. On the Aloff Wallach space $SU(3)/S^1_{k,l}$, the positively curved homogeneous Riemannian metrics constructed in \cite{AW1975} are very standard.
\end{example}

\begin{example}
When
$G/H$ has two non-equivalent irreducible isotropic summands \cite{CN2019}\cite{Ta1999},
any homogeneous $(\alpha_1,\alpha_2)$ metric on $G/H$ is very standard.
\end{example}

\label{subsection-2-4}

\section{Proof of Theorem \ref{main-thm}}

\subsection{Theme of the proof}

All homogeneous manifolds $G/H$ admitting positively curved
homogeneous Riemannian metrics are listed in Section 1. Those manifolds
can be sorted into two categories. On the three Wallach spaces and the Aloff Wallach spaces, positively curved homogeneous Riemannian metrics are constructed in \cite{Wa1972} and \cite{AW1975} respectively. As pointed out in Example \ref{example-1}, they are very standard. On the other manifolds, the normal homogeneous Riemannian metrics are positively curved and very standard. This proves Theorem \ref{main-thm} in one direction.

To
prove Theorem \ref{main-thm} in the other direction, we just need to show that $G/H$ belongs to the list in Section 1.
When $\dim G/H$ is even, it follows
after Theorem 1.3 in \cite{XDHH2015}.

Now suppose that $(G/H,F)$ is an odd dimensional positively curved very standard homogeneous Finsler manifold, and let $\m=\m_1+\cdots+\m_s$ be the corresponding decomposition.
We choose a Cartan subalgebra $\mathfrak{t}$ of $\mathfrak{g}$ such that
$\mathfrak{t}\cap\mathfrak{h}$ is a Cartan subalgebra of $\mathfrak{h}$.
Then the rank inequality $\mathrm{rk}\g\leq\mathrm{rk}\h+1$ (see Theorem
5.2 in \cite{XDHH2015}) implies that $\dim\mathfrak{t}\cap\mathfrak{m}=1$.

Using the given bi-invariant inner product, the roots of $\g$ are viewed as vectors in $\mathfrak{t}$. We will keep this convention throughout this section.
For example, the root system of $\g=sp(n)$
is $$\{\pm e_i\pm e_j,\forall 1\leq i<j\leq n;\pm 2e_i,\forall 1\leq i\leq n \}\subset\mathfrak{t},$$
and the root system of $\g=su(n+1)$ is
$$\{\pm(e_i-e_j),\forall 1\leq i<j\leq n+1\}\subset\mathfrak{t},$$ where $\{e_1,\cdots,e_n\}$ is an orthonormal basis.

The proof is divided into two case, that $F$ is reversible and
that $F$ is non-reversible.

\subsection{The reversible case}
\label{sub-section-the-reversible case}

In this subsection, we further assume $F$ is reversible.
By Corollary 1.2 in \cite{XZ2017}, we only need to verify that $G/H$ can not be one of the following:
\begin{enumerate}
\item $Sp(2)/\mathrm{diag}(z,z^3)\mbox{ with }z\in\mathbb{C}\mbox{ and }|z|=1$;
\item $Sp(2)/\mathrm{diag}(z,z)\mbox{ with }z\in\mathbb{C}\mbox{ and }|z|=1$;
\item $Sp(3)/\mathrm{diag}(z,z,q)\mbox{ with }z\in\mathbb{C}, q\in\mathbb{H}\mbox{ and }|q|=|z|=1$;
\item $SU(4)/\mathrm{diag}(zA,z,\overline{z}^3)\mbox{ with }z\in\mathbb{C},|z|=1\mbox{ and }A\in SU(2)$;
\item $G_2/SU(2)$ with $SU(2)$ the normal subgroup of $SO(4)$ corresponding to the long root.
\end{enumerate}

\begin{lemma}\label{lemma-2}
Let $F$ be a very standard reversible homogeneous Finsler metric on $G/H$,
corresponding to the decomposition $\m=\m_1+\cdots+\m_s$. Then for any $y\in\m_i\backslash\{0\}$, we have
$g_y(u,v)=0$ whenever $u\in\m_j$ for some $j$ and $\langle u,v\rangle =0$.
\end{lemma}

\begin{proof} Using the presentation $F(y)=\sqrt{L(\langle y_1,y_1\rangle ,\cdots,\langle y_s,y_s\rangle )}$
in (\ref{001}) with $y=y_1+\cdots+y_s\in\m\backslash\{0\}$,
 we calculate the fundamental of $F$ as follows. For any $u=u_1+\cdots+u_s$, $v=v_1+\cdots+v_s$, with $u_p,v_p\in\m_p$ for each $p$,
\begin{eqnarray*}
g_y(u,v)&=&\frac12\frac{\partial^2}{\partial s\partial t}|_{s=t=0}
F^2(y+su+tv)\\
&=&\sum_{p=1}^s L_p\langle u_p,v_p\rangle +2\sum_{p=1}^s\sum_{q=1}^s
L_{pq}\langle y_p,u_p\rangle \langle y_q,v_q\rangle .
\end{eqnarray*}
Here $L_p$ and $L_{pq}$ are the values of $\frac{\partial}{\partial t_p}L(t_1,\cdots,t_s)$ and $\frac{\partial^2}{\partial t_p\partial t_q} L(t_1,\cdots,t_s)$ at $(\langle y_1,y_1\rangle ,\cdots,$ $\langle y_s,y_s\rangle )$ respectively.

When $y\in\m_i\backslash\{0\}$ and $u=u_j\in\m_j$ for some $j\neq i$,
$$g_y(u,v)=L_j\langle u,v_j\rangle =L_j\langle u,v\rangle .$$
So $g_y(u,v)=0$ if we further have
$\langle u,v\rangle =0$.

When $y\in\m_i\backslash\{0\}$ and $u=u_i\in\m_i$,
$$g_y(u,v)=L_i\langle u,v\rangle+2L_{ii}\langle y,u\rangle
\langle y,v\rangle.$$
The positive 1-homogeneity of $L$ implies
$$\langle y,y\rangle^2
L_{ii}=\sum_{p=1}^s\sum_{q=1}^s \langle y_p,y_p\rangle\langle y_q,y_q\rangle L_{pq}=0,$$
i.e., $L_{ii}=0$. So we still have $g_y(u,v)=L_i\langle u,v\rangle=0$ when
$\langle u,v\rangle=0$.
This ends the proof of Lemma \ref{lemma-2}.
\end{proof}

As the corollary of Lemma \ref{lemma-1} and Lemma \ref{lemma-2}, we have

\begin{lemma}\label{lemma-3}
Let $F$ be a very standard reversible homogeneous Finsler metric on $G/H$,
corresponding to the decomposition $\m=\m_1+\cdots+\m_s$. For any linearly
independent commuting pair $y\in\m_i$ and $v\in\m_j$,
the flag curvature for $(o,y,y\wedge v)$ vanishes.
\end{lemma}

\begin{proof}The bi-invariance of $\langle\cdot,\cdot\rangle$ and orthogonality between $\h$ and $\m$ implies
$$\langle v,[y,\m]_\m\rangle=\langle v,[y,\m]\rangle=\langle[v,y],\m\rangle=0,$$
so we have $g_y(v,[y,\m])=0$ by Lemma \ref{lemma-2}.
Similar argument also proves
$$g_y(y,[y,\m])=g_y(y,[v,\m])=0.$$
So we get $K(o,y,y\wedge v)=0$ by Lemma \ref{lemma-1}.
\end{proof}

Now we discuss each manifold in the list of the five candidates.

{\bf Case 1}: $G/H=Sp(2)/\mathrm{diag}(z,z^3)$ with $z\in\mathbb{C}$ and $|z|=1$. In this case,
the root plane decomposition of $\g$ is
$$\g=\mathfrak{t}+\g_{\pm 2e_1}+\g_{\pm 2e_2}+\g_{\pm (e_1+e_2)}+\g_{\pm (e_1-e_2)},$$
$\mathfrak{t}=\mathbb{R}e_1+\mathbb{R}e_2$,
$\mathfrak{h}=\mathbb{R}(e_1+3e_2)$, and $\mathfrak{m}=\mathfrak{m}_1+\mathfrak{m}_2+\mathfrak{m}_3+\m_4$, with
$\mathfrak{m}_1=\mathbb{R}(3e_1-e_2)$, $\mathfrak{m}_2=\g_{\pm (e_1-e_2)}+\g_{\pm 2e_1}$, $\m_3=\mathfrak{g}_{\pm (e_1+e_2)}$,
and $\m_4=\g_{\pm 2e_2}$, is the decomposition for a standard homogeneous
Finsler metric on $G/H$.

Since any nonzero $y\in\g_{\pm 2e_1}\subset\m_2$ and any
nonzero $v\in\g_{\pm 2e_2}\subset\m_4$ are a linearly independent commuting pair, Lemma \ref{lemma-3} tells us that $(G/H,F)$ is not positively curved.

{\bf Case 2}: $G/H=Sp(2)/\mathrm{diag}(z,z)\mbox{ with }z\in\mathbb{C}$ and $|z|=1$.
In this case, the root system of $\g$ is $\{\pm 2e_1$, $\pm 2e_2$, $\pm e_1\pm e_2\}$,
$\mathfrak{h}=\mathbb{R}(e_1+e_2)$,
and $\m=\m_1+\m_2$, with $\m_1=\mathbb{R}(e_1-e_2)+\g_{\pm(e_1-e_2)}$
and $\m_2=\g_{\pm(e_1+e_2)}+\g_{\pm 2e_1}+\g_{\pm 2e_2}$.

We can apply Lemma \ref{lemma-3} to the linearly independent commuting pair
$y\in\mathbb{R}(e_1-e_2)\subset\m_1$ and $v\in\g_{\pm(e_1+e_2)}\subset\m_2$, and see that $(G/H,F)$ is not positively curved.

{\bf Case 3}: $G/H=Sp(3)/\mathrm{diag}(z,z,q)\mbox{ with }z\in\mathbb{C}, q\in\mathbb{H}$ and $|q|=|z|=1$. In this case, the root system of $\g$ is
$\{\pm e_i\pm e_j, \forall 1\leq i< j\leq 3; \pm 2e_i,\forall 1\leq i\leq 3
\}$, $\mathfrak{h}=\mathbb{R}(e_1+e_2)+\mathbb{R}e_3+\g_{\pm 2e_3}$,
and $\m=\m_1+\m_2+\m_3$, with $\m_1=\g_{\pm (e_1+e_3)}+\g_{\pm(e_2+e_3)}+\g_{\pm(e_1-e_3)}+\g_{\pm(e_2-e_3)}$,
$\m_2=\g_{\pm(e_1+e_2)}+\g_{\pm2e_1}+\g_{\pm2e_2}$ and
$\m_3=\mathbb{R}(e_1-e_2)+\g_{\pm(e_1-e_2)}$.

We can apply Lemma \ref{lemma-3} to the linearly independent commuting pair
$y\in\g_{\pm2e_1}\subset\m_2$ and $v\in\g_{\pm2e_2}\subset\m_2$, and see that $(G/H,F)$ is not positively curved.

{\bf Case 4}: $SU(4)/\mathrm{diag}(zA,z,\overline{z}^3)\mbox{ with }z\in\mathbb{C}$, $|z|=1$ and $A\in SU(2)$. In this case, the root system of $\g$
is $\{\pm(e_i-e_j),\forall 1\leq i<j\leq 4\}$,
$\h=\mathbb{R}(e_1+e_2+e_3-3e_4)+\mathbb{R}(e_1-e_2)+\g_{\pm(e_1-e_2)}$, and $\m=\m_1+\m_2+\m_3+\m_4$, with
$\m_1=\g_{\pm(e_1-e_3)}+\g_{\pm(e_2-e_3)}$,
$\m_2=\g_{\pm(e_1-e_4)}+\g_{\pm(e_2-e_4)}$,
$\m_3=\g_{\pm(e_3-e_4)}$ and $\m_4=\mathbb{R}(e_1+e_2-2e_3)$.

We can apply Lemma \ref{lemma-3} to the linearly independent commuting pair
$y\in\g_{\pm(e_1-e_3)}\subset\m_1$ and $v\in\g_{\pm(e_2-e_4)}\subset\m_2$,
and see that $(G/H,F)$ is not positively curved.

{\bf Case 5}: $G_2/SU(2)$ with $SU(2)$ the normal subgroup of $SO(4)$ corresponding to the long root. In this case, the root system of $\g$ is
$\{\pm e_1,\pm\sqrt{3}e_2,\pm\frac12e_1\pm\frac{\sqrt{3}}{2}e_2,
\pm\frac32e_1\pm\frac{\sqrt{3}}{2}e_2\}$,  $\h=\mathbb{R}e_2+\g_{\pm\sqrt{3}e_2}$, and $\m=\m_1+\m_2$, with
$\m_1=\g_{\pm(\frac12e_1+\frac{\sqrt{3}}{2}e_2)}+
\g_{\pm(\frac12e_1-\frac{\sqrt{3}}{2}e_2)}+
\g_{\pm(\frac32e_1+\frac{\sqrt{3}}{2}e_2)}+
\g_{\pm(\frac32e_1-\frac{\sqrt{3}}{2}e_2)}$
and $\m_2=\mathbb{R}e_1+\g_{\pm e_1}$.

We can apply Lemma \ref{lemma-3} to the linearly independent commuting pair
$y\in\g_{\pm(\frac12e_1-\frac{\sqrt{3}}{2}e_2)}\subset\m_1$ and $v\in\g_{\pm(\frac32e_1+\frac{\sqrt{3}}{2}e_2)}\subset\m_1$ and see that $(G/H,F)$ is not positively curved.

Above discussion can be summarized as

\begin{theorem}All reversible positively curved very standard homogeneous Finsler manifolds admits positively curved homogeneous Riemannian metrics.
\end{theorem}

\subsection{The non-reversible case}

\label{sub-section-non-reversible}

In this subsection, we further assume $F$ is irreversible.
Then there exists a summand in $\m=\m_1+\cdots+\m_s$, which is assumed to be $\m_s$, satisfying $\dim \m_s=1$ and $\mathfrak{c}_\mathfrak{g}(\h)\cap\m=\mathfrak{n}_\g(\h)\cap\m=\m_s$. We denote by $\mathfrak{h}'=\h\oplus \m_s$ the Lie algebra of the identity component $H'$ of the normalizer $N_G(H)$ of $H$ in $G$ and $\m'=\m_1+\cdots+\m_{s-1}$.

\begin{lemma}\label{lemma-4}
The decomposition $\g=\mathfrak{h}'+\m'$ is orthogonal with respect to $\langle\cdot,\cdot\rangle$, and reductive for the homogeneous manifold $G/H'$. In particular, $\mathrm{Ad}(H')$ acts orthogonally on each $\m_i$  with respect to $\langle\cdot,\cdot\rangle$.
\end{lemma}

\begin{proof}
We only need to show $[\m_s,\m_i]\subset\m_i$ for each $i$.
The other statements in Lemma \ref{lemma-4} are obvious or follows immediately.
For any $u\in\m_s$, $\mathrm{ad}(u):\m_i\rightarrow[u,\m_i]$ is a surjective homomorphism between $\h$-modules. So the irreducible sub $H$-representations in $[u,\m_i]$ are equivalent to that in $\m_i$, which implies $[u,\m_i]\subset\m_i$ by Schur Lemma.
\end{proof}

Using the projection $\mathrm{pr}:\m\rightarrow\m'$, which maps $y=y_1+\cdots+y_s$ with $y_i\in\m_i$ for each $i$ to $y'=y_1+\cdots+y_{s-1}$,
the Minkowski norm $F$ on $\m$ induces a Minkowski norm $F'$ on $\m'$ such that $\mathrm{pr}:(\m,F)\rightarrow(\m',F')$ is a submersion. By Lemma \ref{lemma-4},
both $F$ and $\mathrm{pr}$ are
$\mathrm{Ad}(H')$-invariant, so $F'$, which is uniquely determined by $F$ and $\mathrm{pr}$, is $\mathrm{Ad}(H')$-invariant as well, and
it defines a homogeneous Finsler metric on $G/H'$, such that the canonical projection $\pi:(G/H,F)\rightarrow(G/H',F')$ is a Finsler submersion. Though not necessary,
the $\mathrm{Ad}(H')$-invariance of $F'$ on $\m'$ can be strengthened to the
$O(\m_1)\times\cdots\times O(\m_{s-1})$-invariance, where the orthogonality is with respect to $\langle\cdot,\cdot\rangle$ in each $\m_i$. So $(G/H',F')$ is a reversible very standard homogeneous Finsler manifold.
The following theorem in \cite{AD2001} indicates that $(G/K,F')$ is positively curved.

\begin{theorem}
Let $\pi:(M,F)\rightarrow(M',F')$ be a Finsler submersion. Then any flag $(x',y',\mathbf{P}')$ of $M'$ with $x'=\pi(x)$ can be lifted to
a flag $(x,y,\mathbf{P})$ of $M$, such that $K'(x',y',\mathbf{P}')\geq K(x,y,\mathbf{P})$, where $K'$ and $K$ are flag curvatures for $(M',F')$ and $(M,F)$ respectively.
\end{theorem}

In \cite{XDHH2015}, even dimensional positively curved homogeneous Finsler manifolds are completely classified up to local isometries, or equivalently, in the Lie algebraic level.
Here we only need to concern those $G/H'$ such that $\g$ is simple and $\mathfrak{c}(\mathfrak{h}')$ has a positive dimension, and present them with compact connected simply connected $G$ and connected $H'$. So all possible $G/H'$ and the corresponding $G/H$, respectively, can be listed as follows:
\begin{eqnarray}
G/H'&=&SU(n)/U(n-1),\quad\ \ Sp(n)/U(1)Sp(n-1),\quad SU(3)/T^2,\nonumber\\
G/H&=& SU(n)/SU(n-1),\quad Sp(n)/Sp(n-1),\quad\ \ \ \ \ \ SU(3)/S^1.\label{003}
\end{eqnarray}

The first two homogeneous manifolds in (\ref{003}) are homogeneous spheres,
on which the normal homogeneous Riemannian metrics are very standard and positively curved. For the last one in (\ref{003}), $S^1$ is a closed connected one dimensional subgroup of the maximal torus $T^2$ in $SU(3)$. We may present this $S^1$ as $S^1_{k,l}=\{\mathrm{diag}(z^k,z^l,z^{-k-l}),\forall z\in\mathbb{C},  |z|=1\}$, for some nonzero integers $k$ and $l$. We only need to prove $k+l\neq0$, then $SU(3)/S^1_{k,l}$ is an Aloff Wallach space,
which ends the proof of Theorem \ref{main-thm}.

Assume conversely that $k+l=0$. In this case, the root system of $\mathfrak{g}$ is $\{\pm (e_1-e_2),\pm (e_1-e_3),\pm(e_2-e_3)\}$.
$\mathfrak{h}=\mathbb{R}(e_1-e_2)$, $\m=\m_1+\m_2+\m_3$ with
$\m_1=\g_{\pm(e_1-e_2)}$, $\m_2=\g_{\pm(e_2-e_3)}+\g_{\pm(e_1-e_3)}$,
$\m_3=\mathbb{R}(e_1+e_2-2e_3)$.

\begin{lemma}\label{lemma-5}
Let $F$ be a non-reversible very standard homogeneous Finsler metric on $G/H$ corresponding to the decomposition $\m=\m_1+\cdots+\m_s$ with $\dim \m_s=1$. Then for any nonzero $y\in\m_s$, we have $g_y(u,v)=0$
whenever $u\in\m_j$ for some $j$ and $\langle u,v\rangle=0$.
\end{lemma}

\begin{proof}We present the Minkowski norm $F$ on $\m$ as in (\ref{002}), i.e., $$F(y)=\sqrt{L(\langle y_1,y_1\rangle,\cdots,\langle y_{s-1},y_{s-1}\rangle,y_s)},$$
in which $y=y_1+\cdots+y_s$ with $y_p\in\m_p$ for each $p$, and in particular, $y_s\in\m_s=\mathbb{R}$. Then calculation shows that, for any $u=u_1+\cdots+u_s$ and $v=v_1+\cdots+v_s$, with $u_p,v_p\in\m_p$ for each $i$,
\begin{eqnarray*}
g_y(u,v)&=&\frac12\frac{\partial^2}{\partial s\partial t}|_{s=t=0}
F^2(y+su+tv)\\
&=&2\sum_{p=1}^{s-1}\sum_{q=1}^{s-1}L_{pq}\langle y_p,u_p\rangle
\langle y_q,v_q\rangle+\sum_{p=1}^{s-1}L_{ps}u_s \langle y_p,v_p\rangle
\\
& &+\sum_{p=1}^{s-1}L_{ps}v_s\langle y_p,u_p\rangle+\sum_{p=1}^{s-1}L_p\langle u_p,v_p\rangle+\frac12L_{ss}u_sv_s,
\end{eqnarray*}
in which $L_p$ and $L_{pq}$ are similar to those in the proof of Lemma \ref{lemma-2}. When we have $y\in\m_s\backslash\{0\}$, $u\in\m_j$ for some $j$, and $v$ satisfying $\langle u,v\rangle=0$, we have
$$\langle y_p,u_p\rangle=\langle y_p,v_p\rangle=\langle u_p,v_p\rangle=0,\ \forall 1\leq p\leq s-1,\quad\mbox{and}\quad u_sv_s=0, $$
so $g_y(u,v)=0$ obviously.
\end{proof}

\begin{lemma}\label{lemma-6}
Let $F$ be a non-reversible very standard homogeneous Finsler metric on $G/H$, corresponding to the decomposition $\m=\m_1+\cdots+\m_s$
with $\dim \m_s=1$. Then for any linearly independent commuting pair $y\in\m_s$
and $v\in\m_j$ with $j<s$, the flag curvature for $(o,y,y\wedge v)$ vanishes.
\end{lemma}

\begin{proof}
By Lemma \ref{lemma-5}, $\m_1+\cdots+\m_{s-1}$ is the orthogonal complement of the nonzero $y\in\m_1$, with respect to both $\langle\cdot,\cdot\rangle$
and $g_y(\cdot,\cdot)$. Since we have $\langle y,[y,\m]_\m\rangle=\langle y,[y,\m]\rangle=\langle[y,y],\m\rangle=0$,
we also have $g_y(y,[y,\m]_\m)=0$. Using $[y,v]=0$, we can similarly prove
$g_y(y,[v,\m]_\m)=g_y(v,[y,\m]_\m)=0$. So the flag curvature for $(o,y,y\wedge v)$ vanishes by Lemma \ref{lemma-1}.
\end{proof}

Finally, we apply Lemma \ref{lemma-6} to $G/H=SU(2)/S^1_{k,l}$ with $k+l=0$
and the linearly independent commuting pair $y=e_1+e_2-2e_3\in\m_3$ and $v\in\g_{\pm(e_1-e_2)}\subset\m_1$, and see that this very standard homogeneous Finsler manifold can not be positively curved. This ends the proof of Theorem \ref{main-thm}.

\bigskip

\noindent
{\bf Acknowledgement}\quad
This paper is supported by National Natural Science Foundation of China (No. No. 12131012).

\end{document}